\documentclass{amsart}
\usepackage[T1,T2A]{fontenc}
\usepackage[utf8]{inputenc}
\usepackage{color}

\usepackage{mathrsfs}
\usepackage{amssymb}
\usepackage{enumerate}
\newtheorem{thm}{Theorem}
\newtheorem{lem}[thm]{Lemma}

\newtheorem{conj}[thm]{Conjecture}
\newtheorem{proposition}[thm]{Proposition}
\newtheorem{prop}[thm]{Proposition}

\begin{document}

\title[On uniqueness of distribution]
{On uniqueness of distribution of a random variable whose
independent copies span a subspace in $L_p$}

\author{S. Astashkin}
\address{Samara State University, Pavlova 1, Samara, 443011, Russia}
\email{astash@samsu.ru}

\author{F. Sukochev}
\address{School of Mathematics and Statistics, University of New South Wales, Sydney, 2052, Australia.}
\email{f.sukochev@unsw.edu.au}
\author{D. Zanin}
\address{School of Mathematics and Statistics, University of New South Wales, Sydney, 2052, Australia.}
\email{d.zanin@unsw.edu.au}

\renewcommand{\thefootnote}{\fnsymbol{footnote}}

\thanks{\it Authors acknowledge support from the ARC}

\vspace{-7 mm}

\subjclass[2000]{}


\date{}

\dedicatory{}

\begin{abstract}
Let $1\leq p<2$ and let $L_p=L_p[0,1]$ be the classical $L_p$-space of all (classes of) $p$-integrable functions on $[0,1]$.
It is known that a sequence of independent copies of a mean zero random variable $f\in L_p$
spans in  $L_p$ a subspace isomorphic to some Orlicz sequence space $l_M$.
We present precise connections between $M$ and $f$ and establish conditions under which the distribution of a random variable $f\in L_p$
whose independent copies span $l_M$ in $L_p$ is essentially unique.
\end{abstract}

\maketitle

\bibliographystyle{plain}

\section{Introduction}
It is well known that the class of all subspaces of $L_1=L_1(0,1)$ is very rich and still does not have any
reasonable description. If we consider only symmetric subspaces of $L_1$, that is, subspaces with a symmetric basis or
isomorphs of some symmetric function spaces, then these subspaces are known to be isomorphic to averages of Orlicz spaces \cite{B3, Ray_Sc}.
Far more information is available on subspaces of $L_1$ isomorphic to Orlicz spaces. First of all,
an isomorph of an Orlicz sequence space $l_M\neq l_1$  in  $L_1$  can always be  given by the
span of a sequence of independent identically distributed (i.i.d) random variables.
The latter fact was discovered by  M.I. Kadec  in 1958 \cite{Kad}, who proved that for arbitrary $1\leq p<q<2$ there exists a
symmetrically distributed function $f\in L_p$ ( a $q$-stable  random variable) such that the sequence $\{f_k\}_{k=1}^\infty$ of
independent copies of $f$ spans in $L_p$ a subspace isomorphic to $l_q$.

This direction of study was taken further by J. Bretagnolle and  D. Dacunha-Castelle (see \cite{BD, BD2, B3}). In particular, D. Dacunha-Castelle showed that
for every given mean zero $f\in L_p=L_p(0,1)$, the sequence $\{f_k\}_{k=1}^\infty$ of its independent copies is equivalent in $L_p$ to the unit vector
basis of some Orlicz sequence space $l_M$ \cite[Theorem 1, p.X.8]{B3}. Moreover, J. Bretagnolle and  D. Dacunha-Castelle proved
that an Orlicz function space $L_M=L_M[0,1]$ can be isomorphically embedded into the space $L_p$, $1\leq p<2$, if and only if $M$
is equivalent to a $p$-convex and $2$-concave Orlicz function on $[0,\infty)$ \cite[Theorem IV.3]{BD2}.
Later on some of these results were independently rediscovered by M. Braverman \cite{Br2,Br3}.

Note that the methods used in \cite{BD, BD2, B3, Br2, Br3} depend heavily on the techniques related to the theory of random processes.
In a recent paper \cite{ASorlicz}, two first named co-authors suggested a different approach to study of this problem,
which is based on methods and ideas from the interpolation theory of operators. In addition, it should
be pointed out that papers  \cite{BD, BD2, B3, Br2, Br3}  concern only with the verification of existence of a
function $f$ such that the sequence of its independent copies is equivalent in $L_p$ to the unit vector basis in some
Orlicz sequence space and do not address the question concerning the determination of $f$, whereas \cite{ASorlicz} is
mainly focused on revealing precise connections between the Orlicz function and the distribution of corresponding
random variable $f$. Among other results, in \cite{ASorlicz}, it is shown the following. Let $1\le p<2$ and let
M be a $p$-convex and $2$-concave Orlicz function on $[0,\infty)$ such that $M(t)\not\sim t^p$ for small $t>0$ and
the function
$$
S(u):=-2pM(u)+(p+1)uM'(u)-u^2M''(u)$$
is positive on $(0,\infty),$ increasing and bounded on $(0,1).$ Then, under some
technical conditions on $M$ (see \cite[Proposition~12 and Theorem~15]{ASorlicz}) the unit vector basis in $l_M$
is equivalent in $L_p$ to the sequence $\{f_k\}_{k=1}^\infty$ of independent copies of an arbitrary
mean zero function $f\in L_p$ such that its distribution function
$$
n_f(\tau):=\lambda\{u:|f(u)|>\tau\},\;\; \tau>0$$
($\lambda$ is the Lebesgue measure) is equivalent to the function $S(1/\tau)$ for $\tau\ge 1.$

The present paper continues this direction of research. Our main result
(Theorem \ref{main theorem}) is a somewhat surprising fact
that in the case, when an Orlicz function $M$ is ~\lq far\rq~ from the extreme
functions $t^p$ and $t^2, 1\leq p<2$, the distribution of a random variable $f\in L_p$ whose independent copies span $l_M$
essentially is equivalent to that of the function $$\mathfrak{m}(t)=\frac{1}{M^{-1}(t)},\quad t>0.$$

\begin{thm}\label{main theorem} Let $1\leq p<2$ and let $M$ be a $p-$convex and $2-$concave Orlicz function. The following conditions are equivalent:
\begin{enumerate}[{\rm (i)}]
\item The function $M$ is $(p+\varepsilon)-$convex and $(2-\varepsilon)-$concave for some $\varepsilon>0;$
\item If a sequence $\{f_k\}_{k=1}^{\infty}$ of independent copies of a mean zero random variable  $f\in L_p$ is equivalent in  $L_p$ to
the unit vector basis $\{e_k\}_{k=1}^{\infty}$ in $l_M$, then the distribution function $n_f(\tau)$ is equivalent to that of $\mathfrak{m}$
for large $\tau$.
\item The function $\mathfrak{m}\in L_p$ and any sequence of independent copies of a mean zero random variable equimeasurable with
$\mathfrak{m}$ is equivalent in  $L_p$  to the unit vector basis in $l_M.$
\end{enumerate}
\end{thm}

Observe that even in the simplest  case, when  $1\leq p<q<2$ and $M(t)=t^q, t\ge 0$, the theorem above complements the
above-mentioned classical Kadec result \cite{Kad}, by establishing the uniqueness of the distribution of a mean zero 
random variable $f$ whose
independent copies span $l_q$ in  $L_p$.

It is worth noting that the assertion of Theorem \ref{main theorem} is in a sense sharp. Namely, in Proposition \ref{final prop}
we show that there exist two random variables $x$ and $y$ with non-equivalent distribution for large $\tau$ whose independent copies span in $L_1$ the same
Orlicz space $l_M$, where $M$ is equivalent to the function ${t}/{\log(e/t)}$ for small $t>0$.

Note that in the special case $p=1$, another attempt to describe the connection between the distribution of a random variable $f\in L_p$ and
the corresponding Orlicz function $M$ can be found in \cite{Sch}. However, the methods used in \cite{Sch} have a strong combinatorial flavor and
formulas obtained there seem to be less accessible.  Moreover, in \cite{Sch} the question of uniqueness of distribution of $f$ is not raised at all.

The proof of Theorem \ref{main theorem} is presented in Section \ref{main section}. Two important components of the proof are
Proposition \ref{prelim lemma} and Theorem \ref{preddverie}, which are given in Sections \ref{preliminaries} and \ref{f=f_0}, respectively.

We propose the following conjecture.

\begin{conj} Let $1\le p<2$ and let $M$ be a $p-$convex and $2-$concave Orlicz function. If there is a unique (up to
equivalence near $0$) mean zero function $f$ whose independent copies are equivalent in $L_p$
to the unit vector basis in $l_M,$ then $M$ is $(p+\varepsilon)-$convex and $(2-\varepsilon)-$concave for some $\varepsilon>0.$
\end{conj}

\section{Preliminaries and auxiliary results}\label{preliminaries}

\subsection{Orlicz functions and spaces}
For the theory of Orlicz spaces we refer to \cite{KR61,LT2}.

Let $M$ be an Orlicz function, that is, an increasing convex function on $[0,\infty)$ such
that $M(0)=0.$ To any Orlicz function $M$ we associate the Orlicz sequence space $l_M$ of all sequences of scalars $a=(a_n)_{n=1}^\infty$ such
that
$$\sum_{n=1}^\infty M\left(\frac{|a_n|}{\rho}\right)<\infty$$
for some $\rho>0.$ When equipped with the norm
$$\|a\|_{l_M}:=\inf\left\{\rho>0:\ \sum_{n=1}^\infty M\left(\frac{|a_n|}{\rho}\right)\leq 1\right\},$$
$l_M$ is a Banach space. Clearly, if $M(t)=t^p,$ $p\ge 1,$ then the Orlicz space $l_M$ is the familiar space $l_p.$
Moreover, the sequence $\{e_n\}_{n=1}^\infty$ given by
$$
e_n=(\underbrace{0,\cdots,0,}_{\mbox{$n-1$ times}}1,0,\cdots)$$
is a Schauder basis in every Orlicz space $l_M$ provided that $M$ satisfies
the $\Delta_2-$condition at zero, i.e., there are $u_0>0$ and $C>0$ such that $M(2u)\le CM(u)$ for
all $0<u<u_0.$

Similarly, if $M$ is an Orlicz function, then the Orlicz function space $L_M=L_M[0,1]$ consists of
all measurable functions $x$ on $[0,1]$ such that the norm
$$\|x\|_{L_M}=\inf\Big\{u >0:\;\int\limits_{0}^{1}
M(|x(t)|/u)\,dt\leq 1\Big\}$$
is finite.



Let $1\leq p<q<\infty$. Given an Orlicz function $M$, we say that $M$ is $p$-convex if the map $t \mapsto M(t^{1/p})$ is convex,
and is $q$-concave if the map $t \mapsto M(t^{1/q})$ is concave. Throughout this paper, we  assume that $M(1)=1$ and that
$M:[0,\infty)\to [0,\infty)$ is a bijection.

Careful inspection of the proof of  \cite[Lemma 5]{ASorlicz} establishes the following two lemmas.

\begin{lem}\label{Lemma 20-version 1}
Let $1\leq p<\infty$. An Orlicz function $M:[0,\infty) \to [0,\infty)$ satisfying $\Delta_2$-condition at $0$ is equivalent to a $p$-convex Orlicz function on the
segment $[0,1]$ if and only if there exists a constant $C>0$ such that for all $0<s<1$ and all $0<t\leq 1$ we have
\begin{equation*}\label{p-convex Orlicz functions}
M(st)\le Cs^pM(t).
\end{equation*}
\end{lem}

\begin{lem}\label{Lemma 20-version 2} Let $1< q<\infty$. An Orlicz function $M:[0,\infty) \to [0,\infty)$  is equivalent to a $q$-concave Orlicz function on the
segment $[0,1]$ if and only if there exists a constant $C>0$ such that for all $0<s<1$ and all $0<t\leq 1$ we have
\begin{equation*}\label{p-concave Orlicz functions}
C^{-1}s^qM(t)\le M(st).
\end{equation*}
\end{lem}

%
%
%
In what follows, by $f^*$ we will denote the non-increasing right-continuous rearrangement of a random variable $f$, that is,
$$f^*(s):=\inf\{t: n_f(t)\leq s\},$$
where $n_f$ is the distribution function of the random variable $f$. One says that random variables $f$ and $g$ are equimeasurable if
$f^*(t)=g^*(t),$ $0<t\le 1$ (equivalently, $n_f(\tau)=n_g(\tau)$, $\tau >0$).
Finally, given two positive functions (quasinorms) $f$ and
$g$ are said to be  equivalent (we write $f\sim g$) if there
exists a positive finite constant $C$ such that $C^{-1}f\leq g\leq
Cf$. Sometimes, we say that these functions are equivalent for
large (or small) values of the argument, meaning that the
preceding inequalities hold only for its specified values.




\subsection{A condition for independent copies of a mean zero $f$ to be equivalent in $L_p$ to the unit vector basis
of $l_M$}
For a fixed $f\in L_1(0,1),$ every $k\in{\mathbb N},$ and $t>0$ we set
$$\overline{f}_k(t):=\left\{\begin{aligned}
f(t-k+1),t\in[k-1,k),\\
0, \mbox{ otherwise}.\\\end{aligned}\right.$$

The following assertion is an immediate consequence of the famous Rosenthal inequality \cite{R}
(or, its more general version due to Johnson and Schechtman \cite{JS}). It establishes a connection between
the behaviour in $L_p$ of an arbitrary sequence $\{f_k\}_{k=1}^{\infty}$ of independent copies of a
mean zero random variable $f\in L_p$ and
that of corresponding sequence $\{\overline{f}_k\}_{k=1}^{\infty}$
in the  Banach sum $(L_p+L_2)(0,\infty)$ of the Lebesgue spaces $L_p(0,\infty)$ and $L_2(0,\infty)$.

\begin{lem}\label{our js lemma} Let $1\leq p\leq 2$. For every finitely supported $a=(a_k)_{k=1}^\infty$
and for a mean zero random variable $f\in L_p(0,1)$ we have
$$\Big\|\sum_{k=1}^{\infty}a_kf_k\Big\|_p\sim\Big\|\sum_{k=1}^\infty a_k\overline{f}_k\Big\|_{L_p+L_2}.$$
\end{lem}

Lemma \ref{our js lemma} allows us to investigate sequences of independent identically distributed mean zero random variables
in $L_p=L_p(0,1)$.

\begin{prop}\label{prelim lemma} Let $1\leq p\leq 2$ and let $f\in L_p$ be a mean zero random variable. Then, a sequence $\{f_k\}_{k=1}^{\infty}$ of independent copies of
the random variable $f$ is equivalent (in $L_p$) to the unit vector basis in $l_M$ if and only if
\begin{equation}\label{main f cond}
\frac1{M^{-1}(t)}\sim \left(\frac1t\int_0^tf^*(s)^p\,ds\right)^{1/p}+\left(\frac1t\int_t^1f^*(s)^2\,ds\right)^{1/2},\quad 0<t\le 1.
\end{equation}
\end{prop}
\begin{proof} At first, we assume that a sequence $\{f_k\}_{k=1}^{\infty}$ of independent copies of
$f$ is equivalent in $L_p$ to the unit vector basis in $l_M$. Then, we have
$$\Big\|\sum_{k=1}^{n}e_k\Big\|_{l_M}\sim\Big\|\sum_{k=1}^{n}f_k\Big\|_p
\stackrel{Lemma\ \ref{our js lemma}}{\sim}\Big\|\sum_{k=1}^n \overline{f}_k\Big\|_{L_p+L_2}.$$
Since $1\leq p\leq 2$, it follows that
$$\|x\|_{L_p+L_2}\sim \left(\int_0^1x^*(s)^p\,ds\right)^{1/p}+\left(\int_1^{\infty}x^*(s)^2\,ds\right)^{1/2}.$$
Therefore, from the equalities
$$\Big(\sum_{k=1}^n\overline{f}_k\Big)^*(s)=f^*(\frac{s}{n}),\quad s>0,$$
and
$$\Big\|\sum_{k=1}^{n}e_k\Big\|_{l_M}=\inf\left\{\rho>0: nM(\frac1{\rho})\leq 1\right\}=\frac1{M^{-1}(1/n)},\quad n\geq1,$$
it follows that
$$\frac1{M^{-1}(1/n)}\sim\Big(\int_0^1(f^*(\frac{s}{n}))^p\,ds\Big)^{1/p}+\Big(\int_1^{n}(f^*(\frac{s}{n}))^2\,ds\Big)^{1/2}=$$
$$=\Big(n\int_0^{1/n}(f^*(s))^p\,ds\Big)^{1/p}+\Big(n\int_{1/n}^1(f^*(s))^2\,ds\Big)^{1/2},\quad n\geq 1.$$
Let $t\in(1/(n+1),1/n)$ for some $n\geq 1.$ We clearly have $M^{-1}(1/n)\sim M^{-1}(t)$ and
$$\Big(n\int_0^{1/n}(f^*(s))^p\,ds\Big)^{1/p}+\Big(n\int_{1/n}^1(f^*(s))^2\,ds\Big)^{1/2}\sim$$
$$\sim\Big(\frac1t\int_0^t(f^*(s))^p\,ds\Big)^{1/p}+\Big(\frac1t\int_t^1(f^*(s))^2\,ds\Big)^{1/2}.$$
The assertion \eqref{main f cond} follows immediately from the equivalences  above.


Conversely, by \cite[Theorem 1, p.X.8]{B3} (see also \cite[Theorem~9]{ASorlicz}), for every given
mean zero $f\in L_p(0,1)$ the sequence $\{f_k\}_{k=1}^\infty$ of
independent copies of $f$ is equivalent in $L_p$ to the unit vector basis in some Orlicz sequence space $l_N$. Arguing in the same way as in the first part of the proof, we conclude that
\begin{equation*}
\frac1{N^{-1}(t)}\sim \left(\frac1t\int_0^tf^*(s)^p\,ds\right)^{1/p}+\left(\frac1t\int_t^1f^*(s)^2\,ds\right)^{1/2},\quad t\in(0,1).
\end{equation*}
Taken together with \eqref{main f cond} the equivalence above yields that the Orlicz functions $M$ and $N$ are equivalent on the segment $[0,1]$ and thus, $l_N=l_M$. This completes the proof.
\end{proof}

\section{When does the equivalence \eqref{main f cond} hold for the function $f=\mathfrak{m}$?}\label{f=f_0}

The following proposition provides necessary and sufficient conditions for the function $\mathfrak{m}^p$  to be equivalent to its Cesaro transform.

\begin{proposition}\label{7and9} Let $1\leq p<\infty$ and let $M$ be a $p$-convex Orlicz function satisfying $\Delta_2$-condition at $0$.
The following conditions are equivalent:
\begin{enumerate} [{\rm (i)}]\label{head est up}
\item The function $M$ is equivalent on the segment $[0,1]$ to a $(p+\varepsilon)-$convex Orlicz function 
for some $\varepsilon>0$;
\item\label{head est down} $$\frac1t\int_0^t\mathfrak{m}^p(s)\,ds\leq {\rm const}\cdot \mathfrak{m}^p(t),\quad t\in(0,1).$$
\end{enumerate}
\end{proposition}
\begin{proof} Let the function $\varphi$ be defined by setting $$\varphi(t)=t\mathfrak{m}^p(t), \quad t\in(0,1).$$

${\rm (i)\to (ii)}$. It suffices to show that
\begin{equation}\label{fi first}
\int_0^t\frac{\varphi(s)\,ds}{s}\leq {\rm const}\cdot \varphi(t),\quad t\in(0,1).
\end{equation}

It follows directly from the definitions that, for all $s\in(0,1),$
$$\sup_{0<t\leq 1}\frac{\varphi(st)}{\varphi(t)}=s\cdot\sup_{0<t\leq 1}\left(\frac{(M^{-1}(t))^{p+\varepsilon}}{(M^{-1}(st))^{p+\varepsilon}}\right)^{\frac{p}{p+\varepsilon}}.$$
Since $M$ is $(p+\varepsilon)-$convex, the mapping
$$t\to (M^{-1}(t))^{p+\varepsilon},\quad t\in (0,1],$$
is concave. In particular, we have
$$\frac{(M^{-1}(t))^{p+\varepsilon}}{(M^{-1}(st))^{p+\varepsilon}}\leq s^{-1},\quad 0<s,t\leq 1.$$
Therefore,
$$\sup_{t\in(0,1)}\frac{\varphi(st)}{\varphi(t)}\leq s^{\frac{\varepsilon}{p+\varepsilon}},\quad 0<s\leq 1.$$
Applying now Lemma II.1.4 from \cite{KPS}, we infer \eqref{fi first} and this completes the proof of implication ${\rm (i)\to (ii)}$.

${\rm (ii)\to (i)}$. Since $M$ is $p-$convex, it follows that
$$\frac{M(s)}{s^p}\leq\frac{M(t)}{t^p},\quad 0\leq s\leq t\leq 1.$$
Replacing $s$ with $M^{-1}(s)$ and $t$ with $M^{-1}(t),$ we infer that $\varphi$ is increasing.

By the assumption, we have
$$\int_0^t\frac{\varphi(s)\,ds}{s}\leq C\varphi(t),\quad t\in(0,1),$$
for some $C>0.$ Take $s_0<e^{-2C}.$ We claim that
\begin{equation}\label{h down eq}
\sup_{t\in(0,1)}\frac{\varphi(s_0t)}{\varphi(t)}<1.
\end{equation}
Indeed, suppose that supremum in \eqref{h down eq} equals $1.$ In particular, there exists $t\in(0,1)$ such that $\varphi(s_0t)>\varphi(t)/2.$ Since $\varphi$ is increasing and since $\log(s_0^{-1})>2C,$ it follows that
$$\int_0^t\frac{\varphi(s)\,ds}{s}\geq \int_{s_0t}^t\frac{\varphi(s)\,ds}{s}\geq\varphi(s_0t)\log(\frac{t}{s_0t})>C\varphi(t).$$
This contradiction proves the claim.

According to \eqref{h down eq}, we can fix $a\in(0,1)$ such that
\begin{equation}\label{s0 fi1}
\varphi(s_0t)\leq a\varphi(t),\quad t\in(0,1).
\end{equation}
Without loss of generality, we can assume $a>s_0^{\frac1{1+p}}.$ Hence, there exists $\varepsilon\in(0,1)$ such that $a=s_0^{\frac{\varepsilon}{p+\varepsilon}}.$

For an arbitrary $s\in (0,1]$ there exists $n\in \mathbb{N}$ such that $s\in(s_0^{n+1},s_0^n).$ Since $\varphi$ is increasing, it follows that
$$\varphi(st)\leq\varphi(s_0^nt)\stackrel{\eqref{s0 fi1}}{\leq}s_0^{\frac{n\varepsilon}{p+\varepsilon}}\varphi(t)\leq s_0^{-\frac{\varepsilon}{p+\varepsilon}}s^{\frac{\varepsilon}{p+\varepsilon}}\varphi(t),\quad t\in(0,1).$$
Hence, we have
$$\varphi(st)\leq {\rm const}\cdot s^{\frac{\varepsilon}{p+\varepsilon}}\varphi(t),\quad s,t\in(0,1)$$
or, equivalently,
$$(st)^{-\frac{\varepsilon}{p+\varepsilon}}\varphi(st)\leq {\rm const}\cdot t^{-\frac{\varepsilon}{p+\varepsilon}}\varphi(t),\quad s,t\in(0,1).$$
Therefore, it follows from the definition of $\varphi$ that
$$
{M(st)}\leq  {\rm const}\cdot {s^{p+\varepsilon}}\cdot M(t), \ s,t\in (0,1).
$$
The argument is completed, by referring to Lemma \ref{Lemma 20-version 1}.
\end{proof}

Now, we prove a dual result.

\begin{proposition}\label{8and10} Let  $M$ be a $q$-concave Orlicz function for some $1<q<\infty$. The following conditions are equivalent:
\begin{enumerate} [{\rm (i)}]\label{lemma8}
\item The function $M$ is equivalent to a $(q-\varepsilon)$-concave Orlicz function for some $\varepsilon>0$  on the
segment $[0,1]$;
\item
\begin{equation}\label{lemma10}\frac1t\int_t^1\mathfrak{m}^q(s)\,ds\leq {\rm const}\cdot \mathfrak{m}^q(t),\quad t\in(0,1).
\end{equation}

\end{enumerate}
\end{proposition}
\begin{proof}
 Define the function $\psi$ by setting $$\psi(t):=t\mathfrak{m}^q(t),\quad t\in(0,1).$$
${\rm (i)\to (ii)}$.
It suffices to verify that
$$\int_t^1 \frac{\psi(s)\,ds}{s}\leq {\rm const}\cdot \psi(t),\quad t \in (0,1).
$$
We have
$$\sup\frac{\psi(st)}{\psi(t)}=s\cdot\sup\left(\frac{(M^{-1}(t))^{q-\varepsilon}}{(M^{-1}(st))^{q-\varepsilon}}\right)^{\frac{q}{q-\varepsilon}},$$
where the supremums are taken over all $t\in (0,1)$ and $s>1$ such that $0<st\leq 1$.
Since $M$ is $(q-\varepsilon)-$concave, it follows that the mapping
$$t\to(M^{-1}(t))^{q-\varepsilon},\quad t\in (0,1),$$
is convex. In particular, we have
$$\frac{(M^{-1}(t))^{q-\varepsilon}}{(M^{-1}(st))^{q-\varepsilon}}\leq s^{-1},\quad s>1,\ 0<st\leq 1.$$
Therefore,
$$\sup\frac{\psi(st)}{\psi(t)}\leq s^{-\frac{\varepsilon}{q-\varepsilon}}<1,$$
where again the supremum is taken over all $t\in (0,1)$ and $s>1$ such that $0<st\leq 1$.
Applying now Lemma II.1.5 in \cite{KPS}, we infer \eqref{lemma10}.

${\rm (ii)\to (i)}$. Since $M$ is $q-$concave, it follows that
$$\frac{M(s)}{s^q}\geq\frac{M(t)}{t^q},\quad 0\leq s\leq t\leq 1.$$
Replacing $s$ with $M^{-1}(s)$ and $t$ with $M^{-1}(t),$ we infer that $\psi$ is decreasing.

By the assumption, we have
$$\int_t^1\frac{\psi(s)\,ds}{s}\leq C\psi(t),\quad t\in(0,1),$$
for some $C>0.$ Take $s_0>e^{2C}.$ We claim that
\begin{equation}\label{t down eq}
\sup_{t\in(0,s_0^{-1})}\frac{\psi(s_0t)}{\psi(t)}<1.
\end{equation}
Indeed, suppose that supremum in \eqref{t down eq} equals $1.$ In particular, there exists $t\in(0,s_0^{-1})$ such that $\psi(s_0t)\geq\psi(t)/2.$ Since $\psi$ is decreasing, it follows that
$$\int_t^1\frac{\psi(s)\,ds}{s}\geq \int_t^{s_0t}\frac{\psi(s)\,ds}{s}\geq\psi(s_0t)\log(\frac{s_0t}{t})>C\psi(t).$$
This contradiction proves the claim.

According to \eqref{t down eq}, we can fix $b\in(0,1)$ such that
\begin{equation}\label{s0 fi2}
\psi(s_0t)\leq b\psi(t),\quad t\in(0,s_0^{-1}).
\end{equation}
Without loss of generality, $b>s_0^{-1}.$ Hence, there exists $\varepsilon>0$ such that $b=s_0^{-\frac{\varepsilon}{q-\varepsilon}}.$

Let $s>1$ and $0<t<s^{-1}$. We can find $n\in \mathbb{N}$ such that $s\in(s_0^n,s_0^{n+1})$. Again appealing to the fact that $\psi$ is decreasing, we have
$$\psi(st)\leq\psi(s_0^nt)\stackrel{\eqref{s0 fi2}}{\leq}s_0^{-\frac{n\varepsilon}{q-\varepsilon}}\psi(t)\leq s_0^{\frac{\varepsilon}{q-\varepsilon}}s^{-\frac{\varepsilon}{q-\varepsilon}}\psi(t).$$
It follows that
$$\psi(st)\leq {\rm const}\cdot s^{-\frac{\varepsilon}{q-\varepsilon}}\psi(t),\quad s>1,t\in(0,s^{-1})$$
or, equivalently,
$$s^{\frac{\varepsilon}{q-\varepsilon}}\psi(s)\leq {\rm const}\cdot t^{\frac{\varepsilon}{q-\varepsilon}}\psi(t),\quad 0\leq t\leq s\leq 1.$$
Therefore, from the definition of $\psi,$ we have
$$\frac{s}{M^{-1}(s)^{q-\varepsilon}}\leq {\rm const}\cdot\frac{t}{M^{-1}(t)^{q-\varepsilon}},\quad 0\leq t\leq s\leq 1.$$
or 
$${\rm const}\cdot{s^{q-\varepsilon}}\cdot M(t)\leq M(st)  ,\quad \forall t, s\in (0,1].$$
Applying Lemma \ref{Lemma 20-version 2}, we complete the proof.
\end{proof}

The following theorem answers the question stated in the title of the present section.

\begin{thm}\label{preddverie} Let $1\leq p<2$ and let $M$ be a $p-$convex and $2-$concave Orlicz function. The following conditions are equivalent:
\begin{enumerate}[{\rm (i)}]
\item Equivalence \eqref{main f cond} holds for $f=\mathfrak{m}.$
\item $M$ is $(p+\varepsilon)-$convex and $(2-\varepsilon)-$concave for some $\varepsilon>0.$
\end{enumerate}
\end{thm}
\begin{proof} ${\rm (ii)}\Rightarrow {\rm (i)}$. If $M$ is $(p+\varepsilon)-$convex for some $\varepsilon>0,$ then 
it follows from Proposition \ref{7and9} that
\begin{equation}\label{ss1}
\left(\frac1t\int_0^t\mathfrak{m}^p(s)\,ds\right)^{1/p} \leq {\rm const}\cdot \mathfrak{m}(t),\quad t\in(0,1).
\end{equation}
If $M$ is $(2-\varepsilon)-$concave for some $\varepsilon>0,$ then Proposition \ref{8and10} implies 
\begin{equation}\label{ss2}
\left(\frac1t\int_t^1\mathfrak{m}^2(s)\,ds\right)^{1/2} \leq {\rm const}\cdot \mathfrak{m}(t) ,\quad t\in(0,1).
\end{equation}
Observe now that the inequality
\begin{equation}\label{ss1-converse}
 \mathfrak{m}(t)\leq \left(\frac1t\int_0^t\mathfrak{m}^p(s)\,ds\right)^{1/p}  ,\quad t\in(0,1)
\end{equation}
holds trivially, due to the fact that $\mathfrak{m}$ is decreasing. 
%
The equivalence \eqref{main f cond} for $f=\mathfrak{m}$ follows immediately from \eqref{ss1}, \eqref{ss2} and \eqref{ss1-converse}.

${\rm (i)}\Rightarrow {\rm (ii)}$.  Suppose that \eqref{main f cond} holds for $f=\mathfrak{m}.$
Then, we have \eqref{ss1} and \eqref{ss2}.
%
Applying Propositions \ref{7and9}  and \ref{8and10}, we obtain that $M$ is $(p+\varepsilon)-$convex and $(2-\varepsilon)-$concave
for some $\varepsilon>0,$ and the proof is completed.
\end{proof}

\section{When does equivalence \eqref{main f cond} hold for a unique $f$ (up to equivalence near $0$)?}\label{main section}

This section contains the proof of Theorem \ref{main theorem}.

\begin{proof}[Proof of Theorem \ref{main theorem}]

The implication ${\rm (ii)\to (iii)}$ is obvious and the implication ${\rm (iii)\to (i)}$ follows by combining results of Proposition \ref{prelim lemma} and Theorem \ref{preddverie}.

${\rm (i)\to (ii)}.$ We begin with the following technical lemma.

\begin{lem}\label{large N lemma} Let $1\leq p<\infty$, $1<q<\infty$ and let $M$ be an Orlicz function.
\begin{enumerate}[{\rm (i)}]
\item\label{lb} If $M$ is $(q-\varepsilon)-$concave for some $\varepsilon>0,$ then
$${N}\sup_{t>0}\frac{\mathfrak{m}^q(Nt)}{\mathfrak{m}^q(t)}\to0,\quad N\to\infty.$$
\item\label{la} If $M$ is $(p+\varepsilon)-$convex for some $\varepsilon>0,$ then
$$\frac{1}{N}\cdot\sup_{t>0}\frac{\mathfrak{m}^p(\frac{t}{N})}{\mathfrak{m}^p(t)}\to0,\quad N\to\infty.$$
\end{enumerate}
\end{lem}
\begin{proof} Proofs of \eqref{lb} and \eqref{la} are very similar. So, we prove \eqref{lb} only.

Since $M$ is $(q-\varepsilon)-$concave, it follows that the mapping
$$t\to \frac{M(t)}{t^{q-\varepsilon}},\quad t>0,$$
is decreasing. Hence, the mapping
$$t\to t\mathfrak{m}^{q-\varepsilon}(t)=\frac{t}{(M^{-1}(t))^{q-\varepsilon}},\quad t>0,$$
is also decreasing. Therefore,
$$N^{\frac{q}{q-\varepsilon}}\sup_{t>0}\frac{\mathfrak{m}^q(Nt)}{\mathfrak{m}^q(t)}=\left(\sup_{t>0}\frac{Nt
\mathfrak{m}^{q-\varepsilon}(Nt)}{t\mathfrak{m}^{q-\varepsilon}(t)}\right)^{\frac{q}{q-\varepsilon}}\leq 1,$$
whence
$$N\sup_{t>0}\frac{\mathfrak{m}^q(Nt)}{\mathfrak{m}^q(t)}\leq N^{-\frac{\varepsilon}{q-\varepsilon}}\to0\quad{\rm if}\  N\to\infty.$$
\end{proof}

Now,  let $M$ be a $(p+\varepsilon)-$convex and $(2-\varepsilon)-$concave Orlicz function and let $f$ be a mean zero function from $L_p$.
Suppose that the sequence $\{f_k\}_{k=1}^{\infty}$ of independent copies of  $f$ is equivalent to the unit vector basis $\{e_k\}_{k=1}^{\infty}$ in $l_M$.
It suffices to show that the functions $f^*$ and $\mathfrak{m}$ are equivalent for small values of argument.
For simplicity we  abuse the notation assuming that $f=f^*.$ By Proposition \ref{prelim lemma} we know that the equivalence \eqref{main f cond} holds for $f$, that is,
\begin{equation}\label{main f cond-1}
\mathfrak{m}(t)\sim \left(\frac1t\int_0^tf(s)^p\,ds\right)^{1/p}+\left(\frac1t\int_t^1f(s)^2\,ds\right)^{1/2},\quad t\in(0,1).
\end{equation}
By Theorem \ref{preddverie}, we also have
\begin{equation}\label{main f cond-2}
\mathfrak{m}(t)\sim \left(\frac1t\int_0^t\mathfrak{m}(s)^p\,ds\right)^{1/p}+\left(\frac1t\int_t^1\mathfrak{m}(s)^2\,ds\right)^{1/2},\quad t\in(0,1).
\end{equation}
Observe now that the estimate
\begin{equation}\label{one side}
 f(t)\leq C_1\cdot \mathfrak{m}(t), \quad t\in (0,1),
\end{equation}
for some $C_1>0$ follows immediately from \eqref{main f cond-1} and the (already used) inequality
$$f(t)\leq \left(\frac1t\int_0^tf(s)^p\,ds\right)^{1/p},\;\;t\in (0,1).$$ 
Thus, we need to show that the estimate
\begin{equation}\label{other side}
\mathfrak{m}(t)\leq {\rm const}\cdot f(t),\quad t\in(0,1),
\end{equation}
holds for all sufficiently small $t\in (0,1).$ 
%
By Propositions \ref{7and9} and \ref{8and10}, there exists a constant $C_0>0$ such that
\begin{equation}\label{he1}
\frac1t\int_0^t\mathfrak{m}^p(s)\,ds\leq C_0^p\mathfrak{m}^p(t),\quad t\in(0,1),
\end{equation}
\begin{equation}\label{te1}
\frac1t\int_t^1\mathfrak{m}^2(s)\,ds\leq C_0^2\mathfrak{m}^2(t),\quad t\in(0,1).
\end{equation}
Moreover, there is a constant $C>0$ such that for a given $t\in(0,1),$ from \eqref{main f cond-1} it follows that either
\begin{equation}\label{ass2}
\Big(\frac1t\int_t^1f^2(s)\,ds\Big)^{1/2}\geq\frac1{2C}\mathfrak{m}(t),
\end{equation}
or
\begin{equation}\label{ass1}
\Big(\frac1t\int_0^tf^p(s)\,ds\Big)^{1/p}\geq\frac1{2C}\mathfrak{m}(t).
\end{equation}
By Lemma \ref{large N lemma}, we can fix $N$ so large that
\begin{equation}\label{large N}
\sup_{t>0}\frac{\mathfrak{m}^2(Nt)}{\mathfrak{m}^2(t)}\leq \frac1{8NC^2C_1^2},\quad\sup_{t>0}\frac{\mathfrak{m}^p(\frac{t}{N})}{\mathfrak{m}^p(t)}\leq\frac{N}{2^{p+1}C_1^{p}C^p}.
\end{equation}
Let $t\in(0,1/N).$
Firstly, we consider the situation when \eqref{ass2} holds. Taking squares in this inequality
and then applying \eqref{one side}, we obtain
\begin{eqnarray*}
 \frac1{4C^2}\mathfrak{m}^2(t) &\leq&\frac1t\int_t^1f^2(s)\,ds=\frac1t\int_t^{Nt}f^2(s)\,ds+\frac1t\int_{Nt}^1f^2(s)\,ds\\
&\leq&(N-1)f^2(t)+\frac{NC_1^2}{Nt}\int_{Nt}^1\mathfrak{m}^2(s)\,ds.
\end{eqnarray*}
Hence, by \eqref{te1}, we have
$$\frac1{4C^2}\mathfrak{m}^2(t)\leq (N-1)f^2(t)+NC_1^2C_0^2\mathfrak{m}^2(Nt).$$
Combining the latter estimate with the first inequality in \eqref{large N}, we obtain
$$(N-1)f^2(\frac{t}{N})\geq (N-1)f^2(t)\geq\frac1{4C^2}\mathfrak{m}^2(t)-NC_1^2C_0^2\mathfrak{m}^2(Nt)
\stackrel{\eqref{large N}}{\geq}\frac1{8C^2}\mathfrak{m}^2(t).$$

If \eqref{ass1} holds, then
$$\frac1{2^pC^p}\mathfrak{m}^p(t)\leq\frac1t\int_0^tf^p(s)\,ds=\frac1t\int_0^{t/N}f^p(s)\,ds+\frac1t\int_{t/N}^tf^p(s)\,ds.$$
Taking \eqref{one side} and \eqref{he1} into account, we obtain
\begin{eqnarray*}
 \frac1{2^pC^p}\mathfrak{m}^p(t) &\leq& \frac{C_1^p/N}{t/N}\int_0^{t/N}\mathfrak{m}^p(s)\,ds+(1-\frac1N)f^p(\frac{t}{N})\\
&\leq&\frac1NC_1^pC_0^p\mathfrak{m}^p(\frac{t}{N})+(1-\frac1N)f^p(\frac{t}{N}).
\end{eqnarray*}
We infer from this estimate and the second inequality in \eqref{large N} that
$$(1-\frac1N)f^p(\frac{t}{N})\geq\frac1{2^pC^p}\mathfrak{m}^p(t)-\frac1NC^pC_0^p\mathfrak{m}^p(\frac{t}{N})\stackrel{\eqref{large N}}
{\geq}\frac1{2^{p+1}C^p}\mathfrak{m}^p(t).$$

In either case, we have
$$f(\frac{t}{N})\geq {\rm const}\cdot \mathfrak{m}(t),\quad t\in (0,\frac1N),$$
for a universal constant.
Since $\mathfrak{m}(t)\sim \mathfrak{m}(t/N),$ it follows that
$$f(t)\geq {\rm const}\cdot \mathfrak{m}(t),\quad t\in(0,\frac1{N^2}).$$
The latter inequality together with \eqref{one side} suffices to conclude the proof of implication ${\rm (i)\to (ii)}.$
\end{proof}

\section{Sharpness of Theorem \ref{main theorem}}

Let $\{h_k\}_{k=1}^\infty$ (respectively, $\{g_k\}_{k=1}^\infty$) be a sequence of pairwise disjoint measurable subsets of $(0,1)$ such that
$\lambda(h_k)=2^{-k-2^k}$ (respectively, $\lambda(g_k)=4^{-k-4^k}$), $k\geq 1$.
We define functions $x,y\in L_1(0,1)$ by setting
\begin{equation}\label{def x,y}
x=\sum_{k=1}^\infty 2^{2^k}\chi_{h_k},\quad y=\sum_{k=1}^\infty 4^{4^k}\chi_{g_k},
\end{equation}
($\chi_{c}$ is the indicator function of a set $c$).
\begin{lem}\label{sharpness lemma} We have
$$\int_0^1\min\{x(s),tx^2(s)\}\,ds\sim\int_0^1\min\{y(s),ty^2(s)\}\,ds\sim\frac1{\log(e/t)},\;\;0< t\le 1.$$
\end{lem}
\begin{proof} It is clear that
$$\int_0^1\min\{x(s),tx^2(s)\}\,ds=\sum_{2^{2^k}\geq 1/t}2^{2^k}\cdot 2^{-k-2^k}+t\cdot\sum_{2^{2^k}<1/t}2^{2^{k+1}}\cdot 2^{-k-2^k}.$$
Let $t<1/4.$ If $m$ is the maximal positive integer such that $2^{2^m}<1/t,$ then
$$\int_0^1\min\{x(s),tx^2(s)\}\,ds=\sum_{k=m+1}^{\infty}2^{-k}+t\cdot\sum_{k=1}^m2^{2^k-k}=2^{-m}+t\cdot\sum_{k=1}^m2^{2^k-k}.$$
Also, we have
$$\sum_{k=1}^m2^{2^k-k}\leq 2^{2^m-m}+(m-1)\cdot 2^{2^{m-1}-m+1}\leq 2^{2^m-m}+2^{2^{m-1}}\leq 2\cdot 2^{2^m-m}.$$
Therefore, we obtain
$$2^{-m}\leq\int_0^1\min\{x(s),tx^2(s)\}\,ds\leq 2^{-m}+2t\cdot 2^{2^m-m}\leq 3\cdot 2^{-m}.$$
It follows now from the definition of the number $m$ that
$$\frac1{\log_2(1/t)}\leq\int_0^1\min\{x(s),tx^2(s)\}\,ds\leq\frac6{\log_2(1/t)}.$$
The similar equivalence for $y$ follows {\it mutatis mutandi}.
\end{proof}


\begin{lem} Distributions of the functions $x$ and $y$ are not equivalent.
\end{lem}
\begin{proof} Suppose that $n_x(Ct)\leq Cn_y(t),$ $t>0.$ Fix $k$ such that
$$2^{2k+1}>\log_2 C +1$$
and select $t$ such that both $t$ and $Ct$ belong to the interval $(2^{2^{2k+1}},2^{2^{2k+2}}).$ Then, we have
$$n_x(Ct)=n_x(2^{2^{2k+1}})\geq 2^{-(2k+2)-2^{2k+2}}$$
and
$$n_y(t)=n_y(4^{4^k})\leq 2\cdot 4^{-(k+1)-4^{k+1}}=2^{-2k-1-2^{2k+3}}.$$
It follows from the preceding inequalities that
$$2^{2k+2+2^{2k+2}}\geq\frac1C\cdot 2^{2k+1+2^{2k+3}}$$
or, equivalently,
$$2k+2+2^{2k+2}\geq -\log_2(C)+2k+1+2^{2k+3}.$$
Clearly, the latter inequality contradicts to the choice of $k.$
\end{proof}

Let $\{x_k\}_{k=1}^\infty$ (respectively, $\{y_k\}_{k=1}^\infty$) be a sequence of independent copies of a mean zero 
random variable equimeasurable with $x$ (respectively, $y$),
where $x$ and $y$ are defined in \eqref{def x,y}. Let us show that the sequences $\{x_k\}_{k=1}^\infty$ and $\{y_k\}_{k=1}^\infty$ span in $L_1$ the
same Orlicz space $l_M$, where $M$ is equivalent to the function ${t}/{\log(e/t)}$ for small $t>0$. Note that $M$ does not satisfy condition (i) of Theorem \ref{main theorem};
more precisely, $M$ is not $(1+\varepsilon)$-convex for any $\varepsilon>0$. Taking into account Lemma \ref{our js lemma}, it suffices to prove the following proposition.
\begin{prop}\label{final prop} For every finitely supported $a=(a_k)_{k=1}^\infty,$ we have
$$\Big\|\sum_{k=1}^na_k\overline{x}_k\Big\|_{L_1+L_2}\sim\Big\|\sum_{k=1}^na_k\overline{y}_k\Big\|_{L_1+L_2}\sim\|(a_k)_{k=1}^\infty\|_{l_M}.$$
\end{prop}
\begin{proof} Define the Orlicz function $N$ by setting
$$N(t)=
\begin{cases}
t^2,\quad t\in(0,1)\\
2t-1,\quad t\geq 1.
\end{cases}
$$
It is easy to check that $\|z\|_{L_1+L_2}\sim\|z\|_{L_N}$ for every $z\in L_1+L_2,$ where
$L_N$ is the function Orlicz space on $[0,1].$

Setting
$$M(t)=\int_0^1N(tx(s))\,ds,\quad t>0,$$
we obtain
\begin{gather*}
\begin{split}
\|\sum_{k=1}^\infty a_k\overline{x}_k\|_{L_N}\leq 1&\Longleftrightarrow
\int_0^\infty N(\sum_{k=1}^\infty|a_k||\overline{x}_k(s)|)\,ds\leq 1\\
&\Longleftrightarrow \sum_{k=1}^\infty \int_0^1N(|a_k||x_k(s)|)\,ds\leq 1\\
&\Longleftrightarrow \sum_{k=1}^\infty M(a_k)\leq 1\Longleftrightarrow \|a\|_{l_M}\leq 1.
\end{split}
\end{gather*}
Therefore,
$$\|\sum_{k=1}^\infty a_k\overline{x}_k\|_{L_1+L_2}\sim\|a\|_{l_M}.$$
Since $N(t)\sim\min\{t,t^2\}$ $(t>0),$ it follows that
$$M(t)\sim\int_0^1\min\{tx(s),(tx(s))^2\}\,ds,$$
and from Lemma \ref{sharpness lemma} it follows that
$$M(t)\sim \frac{t}{\log(e/t)},\quad 0<t\le 1.$$
This proves the assertion for the sequence $\{x_k\}$. The proof of the similar assertion for $\{y_k\}$ is the same.
\end{proof}

\end{document}